\documentclass[10pt,epsfig]{article}         
\usepackage[parfill]{parskip}    % Activate to begin paragraphs with an empty line rather than an indent
\usepackage{enumerate,picture}
\usepackage{amssymb,latexsym,amsmath,amsxtra,amsthm,amsfonts,amscd,slashed,mathrsfs,yhmath}
\usepackage[pdftex]{graphicx}
\usepackage{graphics}

\theoremstyle{plain}
\newtheorem*{BI}{Theorem A}
\newtheorem*{HL}{Theorem B}
\newtheorem{lemma}{Lemma}
\newtheorem{theorem}{Theorem}
\newtheorem{prop}{Proposition}
\newtheorem{cor}{Corollary}
\theoremstyle{definition}
\newtheorem{defn}{Definition}
\newtheorem{rmk}{Remark}
\newtheorem{eg}{Example}

\makeatletter
\def\thm@space@setup{
  \thm@preskip=\parskip \thm@postskip=0pt
}
\makeatother

\title{Banach's Indicatrix Reloaded}
\author{Ikemefuna Agbanusi}
\date{}                                           % Activate to display a given date or no date

\newcommand{\paren}[1]{\left(#1\right)}

\newcommand{\abs}[1]{\left|#1\right|}
\newcommand{\brac}[1]{\left[#1\right]}

\newcommand{\disp}{\displaystyle}
\renewcommand{\vec}[1]{\mathbf{#1}}

\DeclareMathOperator{\dist}{dist}

\DeclareMathOperator{\mes}{meas}

\def\N{\mathbb{N}} 
\def\R{\mathbb{R}}

\begin{document}
\maketitle

\begin{abstract} 
Banach famously related the smoothness of a function to the size of its level sets. More precisely, he showed that a continuous function is of bounded variation exactly when its ``indicatrix" is integrable. In a similar vein, we connect the smoothness of a function---measured now by its integral modulus of continuity---to the structure of its \emph{superlevel} sets. Our approach ultimately rests on a continuum incidence problem for  quantifying the regularity of open sets.
The pay off is a refinement of Banach's original theorem and an answer to a question of Garsia--Sawyer.
\end{abstract}

\section{Background}
The purpose of this paper is to further study the relationship between the level sets of a continuous function and its $L^p$ modulus of continuity. The genesis of this problem is Banach's ``indicatrix" theorem which we now recall.

Let $f(x)$ be a continuous 1-periodic real valued function normalized to satisfy $0\leq f(x)\leq1$. Let $n(y)$ count the number (possibly infinite) of solutions to $f(x)=y$. Following {\sc Natanson}  \cite[p.\,225]{Natanson:1955bl}, we refer to the function $n(y)$ as the \emph{Banach indicatrix}. It was introduced by {\sc Banach} \cite{Banach:1925to} nearly a century ago and used to establish the following remarkable result.
\begin{BI}[Banach's Indicatrix Theorem]\label{thm:BI}
A continuous function is of bounded variation if and only if  $n(y)$ is Lebesgue integrable. 
\end{BI}
Recall that a function is said to be of \emph{bounded variation}, or simply BV, if
\[
V(f):=\sup_\mathcal{P}\sum_k|f(x_{k+1})-f(x_k)|<\infty.
\]
Here the supremum is taken over all finite partitions or subdivisions of $[0,1]$.
Proofs of the indicatrix theorem and basic facts about BV functions can also be found in {\sc Saks} \cite[p.\,280]{Saks:1937wo},  {\sc Federer} \cite[p.\,423--426]{Federer:1944gn} and {\sc Cesari} \cite[p.\,321]{Cesari:1958iz}.

Shortly after, {\sc Hardy--Littlewood} \cite[p.\,619]{Hardy:1928hw} gave a different characterization of BV functions. Theirs uses the $L^p$ modulus of continuity defined by
\[\omega(f,t)_p:=\sup_{|h|\leq t}\paren{\int_0^1|f(x+h)-f(x)|^p\,dx}^\frac1p.\] 
\begin{HL}[Hardy--Littlewood Characterization]
A measurable function $f$ is of bounded variation if and only if $\omega(f,t)_1=O(t)$.
\end{HL}
These theorems show that the conditions $n(y)\in L^1([0,1])$ and $\omega(f,t)_1=O(t)$ are actually equivalent for continuous functions. They also hint at an underlying link between the modulus of continuity and the indicatrix, so it is natural to wonder how deep the connection goes. In particular, what can be said about  $\omega(f,t)_1$ when $n(y)\notin L^1([0,1])$? This puzzle is our main preoccupation.

Similar questions were  also raised by {\sc Garsia--Sawyer} \cite[p.\,589--591]{Garsia:1964ss}. They pointed out (without proof) that 
\begin{equation}\label{eqn:GS1}
\int_0^1\log n(y)\,dy<\infty\text{\,\,implies that\,\,} \omega(f,t)_1=O\paren{|\log t|^{-1}},
\end{equation}
whereas
\begin{equation}\label{eqn:GS2}
\int_0^1(\log n(y))^{\frac12}\,dy<\infty\text{\,\,implies that\,\,} \omega(f,t)_1=O\big(|\log t|^{-\frac12}\big).
\end{equation}
As reported in \cite{Garsia:1964ss}, earlier {\sc Beurling} (unpublished) had shown that $(n(y))^\frac1p\in L^1$ implies that $f$ has bounded $p$-variation when $p>1$. A corollary (see Remark \ref{rmk:alt_BVp})  is the fact that
\begin{equation}\label{eqn:indicatrix_BVp}
\int_0^1 (n(y))^\frac1p\,dy<\infty\text{\,\,implies\,\,} \omega(f,t)_p=O(t^\frac1p).
\end{equation}
Our contribution here is a unified and elementary approach to these results. It turns out they are special cases of a theorem relating the $L^p$ modulus of continuity and the geometry of \emph{superlevel} sets. This story is outlined next.
%\footnote{ I strongly encourage you to set this paper aside and try formulating such a result.}.
\section{Results}
First we establish some notation. For fixed $y$, define the \emph{superlevel} set $E_y=\{x:f(x)>y\}$.
Since $f$ is continuous and 1-periodic, $E_y$ is actually an open subset of the unit circle. We typically identify the circle with the interval $[0,1)$ having its endpoints ``glued". By a standard theorem, $E_y$ is at most a countable union, $\bigcup_k I_{y,k}$, of disjoint open intervals (arcs), $I_{y,k}$. Throughout, $l_{y,k}$ refers to the length of the intervals $I_{y,k}$ in $E_y$.

{\sc Garsia--Sawyer} \cite[p.\,589]{Garsia:1964ss} introduced the function $N(E_y)$ which counts the number of intervals in this collection. The indicatrices, $n(y)$ and $N(E_y)$, are virtually equivalent since $n(y)=2N(E_y)$ for almost every $y\in[0,1]$ as shown in \cite[p.\,595]{Garsia:1964ss}. We switch between the two without much ado.

Our generalization of \eqref{eqn:GS1} and \eqref{eqn:GS2} involves a certain class of functions. To our knowledge, this class has not been singled out in the literature, but it appears related to the ``gauge functions" which define Hausdorff measures.

\begin{defn}
The class $\Phi$ consists of functions, $\phi: (0,1]\to\R$, satisfying
\begin{enumerate}[(i)]
\item $\phi(t)$ is non-increasing 
\item  $t\phi(t)$ is non-decreasing, bounded and concave on some interval $(0,\frac1c)$.
\end{enumerate} 
\end{defn}
Examples of functions in $\Phi$ are $t^{-\alpha}$ and $(\log\frac1t)^\alpha$ with $0\leq\alpha\leq1$. Another is 
$t^{-\alpha}\exp(\gamma|\log t|^\beta)$ with $0\leq\alpha,\beta<1$ and $\gamma>0$.  One can cook up more complicated examples combining negative powers of $t$, logarithms, iterated logarithms, and so on. 

With all the pieces in place, we can finally state something precise.

\begin{theorem}\label{thm:intro_thm}
If $\phi(t)\in\Phi$, $\int_0^1\sum_kl_{y,k}\phi(l_{y,k})\,dy<\infty, \text{\,\,then\,\,\,}\omega(f,t)_1=O\paren{1\big/\phi(t)}$. In particular, $\int_0^1|E_y|\,\phi\paren{|E_y|\big/N(E_y)}\,dy<\infty$ implies $\omega(f,t)_1=O\paren{1\big/\phi(t)}$.
\end{theorem}

Note that the choice $\phi(t)=\frac1t$ in the second half of Theorem \ref{thm:intro_thm} yields one of the implications in Banach's original theorem, while $\phi(t)=\log\frac1t$ and $\phi(t)=(\log\frac1t)^\frac12$ recover  Garsia and Sawyer's observations \eqref{eqn:GS1} and \eqref{eqn:GS2} respectively. Moreover, a careful parsing of Theorem \ref{thm:intro_thm} shows that it provides much weaker sufficient conditions for the conclusions of \eqref{eqn:GS1} and \eqref{eqn:GS2} to hold.

It is also  worth pointing out that the sum, $\sum_k l_k\phi(l_k)$, appearing in Theorem \ref{thm:intro_thm} actually encodes some of the geometry of the open set $E=\bigcup I_k$.  For example, if $\phi(t)=1$, the sum gives  the measure $|E|$.  With $\phi(t)=\frac1t$, it yields $N(E)$, the number of intervals in $E$.  For $\phi(t)=\log\frac1t$, the sum corresponds to the \emph{entropy} or {\sc Bary's} \cite[p.\,421]{Bary:1964sj} logarithmic measure, and so on. One way to summarize Theorem \ref{thm:intro_thm} is that it gives analytic information about a function by synthesizing geometric data.

Having said this, many properties are \textbf{not} encoded in the length sequence $\{l_k\}$. Indeed, there are lots of configurations of intervals with the same prescribed lengths. With different arrangements, the complement $E^c$ may be countable, uncountable, nowhere dense, or even the disjoint union of closed intervals. In the last case, the precise location of these ``gaps"  is important. For these (and other!) reasons, Theorem \ref{thm:intro_thm} is likely not the final word on this problem.

\subsection{Map and Highlights}
Here are some of the key ideas and a few token results. In \S3 we reduce matters to an incidence problem for translates of an open set with its complement. The identification and partial solution of this problem is really this paper's main contribution. This machinery is applied to the superlevel sets $E_y$ leading to a proof of Theorem \ref{thm:intro_thm}. Another striking fact proved there is that $\omega(\chi_E,t)_p=O\big(t^{\frac{1}{p}(1-d_X)-\epsilon}\big)$ for any $\epsilon>0$. Here $d_X$ is the Besicovitch--Taylor dimension index of $E$.

In \S4 we test our methods on specific examples which provide concrete estimates and comparisons amongst the various results.
For instance, if $E$ is the complement of a ``Fat Cantor set", we show that $\omega(\chi_E,t)_p=O\big(t^{\frac{1}{p}(1-d_X)}\big)$. This improves the ``$\epsilon$ loss" above but we do not know if it is sharp.

Multidimensional variants of our results are given in \S5. The treatment there is brief and focuses mainly on generalizing the incidence problem. Finally, \S6 speculates on where and how deep this rabbit hole goes.

I think you will become as enamored as I am with this problem so let us get started.

\section{Proofs}
The bulk of our effort is directed at studying characteristic functions of open sets. To understand why, observe that the ``layer cake" formula $f(x)=\int_0^1\chi_{E_y}(x)\,dy$ implies that
\[
\paren{\int_0^1|f(x+h)-f(x)|^p\,dx}^{\frac1p}=\paren{\int_0^1\abs{\int_0^1(\chi_{E_y}(x+h)-\chi_{E_y}(x))\,dy}^p\,dx}^{\frac1p}.
\]
Minkowski's inequality and the bound $\abs{\chi_{E_y}(x+h)-\chi_{E_y}(x)}\leq1$ then give
\begin{align*}
\paren{\int_0^1|f(x+h)-f(x)|^p\,dx}^{\frac1p}
&\leq\int_0^1\paren{\int_0^1|\chi_{E_y}(x+h)-\chi_{E_y}(x)|\,dx}^{\frac1p}\,dy.
\end{align*}
The inner integral is  essentially the $L^1$ modulus of continuity of $\chi_{E_y}(x)$. The next definition will help make this more precise.
\begin{defn}\label{def:tau}
Let $E$ be an open subset of the circle and let $K$ denote its complement. For $h>0$, define the function
\begin{align*}
\tau(h,E) &=\mes\{x:|\chi_E(x+h)-\chi_E(x)|\neq0\}\\
&=\mes\{x\in E:x+h\in K\} +\mes\{x\in E:x-h\in K\}.
\end{align*}
\end{defn}
It is fairly straightforward to check that $\disp\omega(\chi_E,t)_p=\sup_{|h|\leq t}(\tau(h,E))^\frac1p$, and our earlier computation can now be summed up rather nicely.
\begin{prop}\label{prop:main_reduc}
For $p\geq1$, we have $\disp\omega(f,t)_p\leq\int_0^1\big(\sup_{h\leq t} \tau(h,E_y)\big)^{\frac1p}\,dy$.
\end{prop}

Evidently, it all boils down to finding estimates for $\tau(h,E)$. From the definition, it is clear that estimating $\tau(h,E)$ requires solving an incidence problem. This is the problem of determining when translates of an arbitrary open set intersect with its complement. Intuitively, this seems related to measuring the ``proximity" of points in $E$ to its complement. 

Let us dispatch some easy but important cases. First, we have the obvious bound, $\tau(h,E)\leq2\min\{|E|,|K|\}$. Thus $\tau(h,E)=0$ if either $E$ or $K$ has measure zero.

The other easy case is when $N(E)$ is finite since we have $\tau(h,E)\leq 2hN(E)$. This is the situation underlying the Banach and Hardy--Littlewood theorems. The example of $E=\bigcup_{k=1}^N(a_k,b_k)$ with $0<a_1<b_1<a_2<b_2<\ldots<a_N<b_N<1$ and $h$ small enough (i.e. $ h\leq\min\{a_1, 1-b_N,b_k-a_k, a_{k+1}-b_k\}$ or $h$ smaller than the lengths of all the intervals and ``gaps" between them) shows this bound is sharp for this particular case.

It turns out that we can ``interpolate" between these cases to get a partial solution of the incidence problem. This very simple result is the real workhorse of our paper.
\begin{lemma}\label{lem:main_lem}
If $E=\bigcup_kI_k$, and $\{l_k\}$ is its length sequence, then
\begin{equation}\label{eqn:main_sum}
\tau(h,E) \leq2\sum_{l_k\geq h} h+2\sum_{l_k<h} l_k.
\end{equation}
\end{lemma}
\begin{proof}
 It follows from Definition \ref{def:tau} that 
\[
\tau(h,E) =\sum_k\paren{\mes\{x\in I_k:x+h\in K\} +\mes\{x\in I_k:x-h\in K\}}
\]
We split the sum depending on whether $l_k\geq h$. If $I_k=(a_k,b_k)$ with $b_k-a_k\geq h$ it follows that $\{x\in I_k:x+h\in K\} \subset (b_k-h,b_k)$ while $\{x\in I_k:x-h\in K\} \subset (a_k,a_k+h)$.  Hence if $l_k\geq h$ we see that $|\{x\in I_k:x+h\in K\}|+|\{x\in I_k:x-h\in K\}|\leq 2h$. On the other hand if $l_k<h$ we use the trivial bound $|\{x\in I_k:x\pm h\in K\}|\leq l_k$ and the proof is complete.
\end{proof}

Before applying this Lemma, we comment briefly on its accuracy and the deeper idea behind it.
First define the set $K(h)=\{x:\dist(x,K)\leq h\}$. If $x$ is in $E$ but $x+h$ is in $K$, then $\dist(x,x+h)=h$ and thus $x$ belongs to $K(h)\backslash K$. Similarly, $x\in E$ and $x-h\in K$ also implies $x\in K(h)\backslash K$. From this follows the inclusions 
\begin{equation}\label{eqn:incl}
\{x:|\chi_E(x+h)-\chi_E(x)|\neq0\}\subset K(h)\backslash K\subset K(h).
\end{equation}
Lemma \ref{lem:main_lem} essentially estimates the measure of the middle set, $K(h)\backslash K$. This is likely not sharp since, in principle, the measure of all three sets could differ. 
Regardless, Lemma \ref{lem:main_lem} yields a preliminary version of Theorem \ref{thm:intro_thm}.
\begin{theorem}\label{thm:main_thm}
If $\phi\in\Phi$ and $\sum_{k} l_k\phi(l_k)\leq M<\infty$ then, as $h\to0$,  $\tau(h,E)\leq A/\phi(h)$ for some $A>0$ depending only on $M$ and $\phi$.
\end{theorem}
\begin{proof}
Since $\phi\in\Phi$, $\phi(l_k)\geq\phi(h)$ if $l_k<h$ while $l_k\phi(l_k)\geq h\phi(h)$ if  $\frac1c> l_k\geq h$. These observations applied to Lemma \ref{lem:main_lem} lead to the chain of inequalities:
\begin{align*}
\sum_{l_k<h} l_k+\sum_{l_k\geq h}h&\leq \sum_{l_k<h} l_k+\sum_{l_k\geq \frac1c}h+\sum_{h\leq l_k< \frac1c}h \\
&\leq \frac{1}{\phi(h)}\sum_{l_k<h} l_k\phi(l_k)+h\sum_{l_k\geq \frac1c}1+\frac{1}{\phi(h)}\sum_{h\leq l_k< \frac1c}l_k\phi(l_k).
\end{align*}
From this, and the boundedness $t\phi(t)$, follows
\[
\tau(h,E) \leq\frac{2}{\phi(h)}\Big(ch\phi(h)+\sum_{k} l_k\phi(l_k)\Big)\leq\frac{A}{\phi(h)},
\]
concluding the proof.
\end{proof}
Let us sketch the proofs of Theorem \ref{thm:intro_thm} and \eqref{eqn:indicatrix_BVp}.
\begin{proof}[Proof of Theorem \ref{thm:intro_thm} and \eqref{eqn:indicatrix_BVp}]
If $B$ is an upper bound for $ch\phi(h)$, Proposition \ref{prop:main_reduc} and Theorem \ref{thm:main_thm} applied to the family of open sets $E_y$ show that
\[
\omega(f,t)_1\leq\int_0^1\sup_{h\leq t}\frac{2}{\phi(h)}\big(B+\sum_{k} l_{y,k}\phi(l_{y,k})\Big)\,dy\leq\frac{2}{\phi(t)}\Big(B+\int_0^1\sum_{k} l_{y,k}\phi(l_{y,k})\,dy\Big).
\]
To obtain the last inequality, we recall  that $\phi$ is non-increasing which implies
\[\sup_{h\leq t}(1\big/\phi(h))\leq1\big/\phi(t).\]
The integrand is finite by assumption and the first part of Theorem \ref{thm:intro_thm} is proved.
For the second part of Theorem \ref{thm:intro_thm}, observe that if $N(E_y)<\infty$, the concavity of $t\phi(t)$ and Jensen's inequality imply
\[
\sum_{k} l_{y,k}\phi(l_{y,k})\leq |E_y|\phi\paren{|E_y|/N(E_y)},
\]
which essentially completes the proof of Theorem \ref{thm:intro_thm}.

To prove \eqref{eqn:indicatrix_BVp}, we again use Proposition \ref{prop:main_reduc}. The key observation now is that for $N(E_y)$ finite we have $\tau(h,E_y)\leq2hN(E_y)$.
Hence
\[
\omega(f,t)_p\leq\int_0^1(\sup_{h\leq t}2hN(E_y))^{\frac1p}\,dy\leq t^\frac1p\int_0^1(n(y))^{\frac1p}\,dy,
\]
which implies \eqref{eqn:indicatrix_BVp}.
\end{proof}

\begin{rmk}\label{rmk:alt_BVp}
We outline a different proof of \eqref{eqn:indicatrix_BVp}. It combines two facts:

(i) $(n(y))^\frac1p\in L^1([0,1])$ implies $f$ is of bounded $p$--variation. This is stated in {\sc Garsia--Sawyer} \cite{Garsia:1964ss} and a proof is given in {\sc Asatiani--Chanturia} \cite[p.\,52]{Asatiani:1973lh}. This roughly corresponds to Theorem A.

(ii) If $f$ is of bounded $p$--variation then $\omega(f,t)_p=O(t^\frac1p)$. For this fact and the definition of $p$--variation, consult {\sc Young} \cite[p.\,259]{Young:1936xw}, and  {\sc Gehring} \cite[p.\,424]{Gehring:1954wq}. This corresponds to Theorem B. The appropriate converse can be found in {\sc Terekhin} \cite[p.\,661]{Terekhin:1967ss} and {\sc Kolyada--Lind} \cite[p.\,589]{Kolyada:2009xu}. 
\qed
\end{rmk}

A drawback of the arguments given so far is that they do not specify the function $\phi(t)$. One remedy is to select the ``power type" family $\phi(t)=t^{\beta-1}$ and to pick $\beta$ in some optimal fashion depending on $E$. For this we introduce the following ``dimension index" first defined by {\sc Besicovitch--Taylor} \cite{Besicovitch:1954vg}.
\begin{defn}
The \emph{Besicovitch--Taylor index}, $d_X(E)$, is defined by
\[d_X=\inf\{\beta:\sum l_k^\beta<\infty\},\] 
where $E=\bigcup I_k$ and $\{l_k\}$ is its length sequence.
\end{defn}
The notation $d_X$ is from {\sc Tricot}  \cite{Tricot:1986hp} who also christened it the \emph{exchange index}.
With this notion the following is an immediate consequence of Theorem \ref{thm:main_thm}. 
\begin{cor}\label{cor:BT}
We have $\tau(h,E)=O\paren{h^{\alpha}}$, for all $0\leq\alpha<1-d_X(E)$.
\end{cor}
The inclusion of the endpoint, $1-d_X$, in Corollary \ref{cor:BT} depends on the open set being considered. When this endpoint is not included, we need to either (i) fine tune the choice for $\phi$, or (ii) perform the optimization with a different family of functions altogether, or (iii) employ the more refined Lemma \ref{lem:main_lem} directly. This issue is explored further in \S4.

\section{Examples}
In this section, we get to see the ideas of \S3 in action\footnote{These examples can be viewed as exercises for the author.}. We warn that none of the figures are drawn to scale.
\begin{eg}[The Textbook Example]
The standard continuous function that is not of bounded variation is $\disp g(x)=\abs{x\sin\frac\pi x}$ for $0<x\leq1$ with $g(0)=0$. We will compute the $L^1$ modulus of continuity of this function to see how close it comes to being of bounded variation. To make the algebra easier we  perform the computation for a different but closely related function. 

\begin{minipage}{0.5\textwidth}
The construction of this function can be found in {\sc Pierpont} \cite[p.\,351]{Pierpont:1905os}. Fix any $b>1$ and define $f(0)=f(b)=0$, and $f(1/k)=1/k$ for $k\in\N$. Pick $1/(k+1)<a_k<1/k$ (the midpoint is a good choice) and define $f(a_k)=0$ as well. The function is defined by linear interpolation between successive points everywhere else on $(0,b)$.
\end{minipage}
\qquad
\begin{minipage}{0.5\textwidth}
\includegraphics[scale=.575]{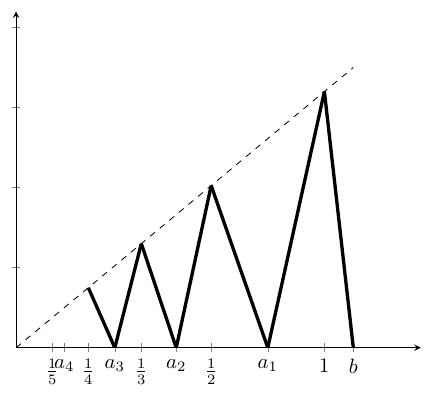}
\end{minipage}
We use the different results in turn. This should clarify their relative strengths.
\begin{enumerate}[(i)]
\item For $1/(k+1)<y<1/k$  we see that $n(y)=2k$ and $N(E_y)=k$. Hence
\[
\int_0^1n(y)\,dy =\sum_{k=1}^\infty\int_{\frac{1}{k+1}}^{\frac1k }n(y)\,dy =2 \sum_{k=2}^\infty k^{-1}=\infty,
\]
and so $f$ is not of bounded variation by Banach's result, i.e. Theorem A. 
\item We now apply Theorem \ref{thm:intro_thm} with $\phi(t)=t^{-\alpha}$ and $0\leq\alpha<1$. Here
\[
\int_0^1|E_y|\,\phi\paren{|E_y|\big/N(E_y)}\,dy=\int_0^1|E_y|\paren{|E_y|\big/N(E_y)}^{-\alpha}\,dy\leq\int_0^1\paren{N(E_y)}^{\alpha}\,dy.
\]
A similar computation to the one earlier shows
\[
\int_0^1\paren{N(E_y)}^{\alpha}\,dy\leq\sum_{k=1}k^{\alpha-2}.
\]
If $0\leq\alpha<1$, this sum converges and so $\omega(f,t)_1=O(t^\alpha)$.  Hence $f$ just \emph{barely fails} being of bounded variation. Incidentally, we have also shown that for $p>1$, $(N(E_y))^{\frac1p}$ is integrable. Thus for $p>1$, $f$ is of bounded $p$--variation and $\omega(f,t)_p=O(t^\frac1p)$.
  \item We go a step further and apply Theorem \ref{thm:intro_thm} with $\phi(t)=t^{-1}|\ln t|^{-\beta}$. Now
\[
\int_0^1|E_y|\,\phi\paren{\frac{|E_y|}{N(E_y)}}dy=\int_0^1\frac{N(E_y)}{\big(\ln(N(E_y)\big/|E_y|)\big)^\beta}dy=\sum_{k=1}^\infty k\int\limits_{\frac{1}{k+1}}^{\frac1k }\frac{dy}{\big(\ln(k\big/|E_y|)\big)^\beta}.
\]
For large $k$, the summand is roughly $\dfrac{1}{k(\ln k)^\beta}$, and this sum converges if $\beta>1$, but diverges otherwise. Thus, $\omega(f,t)_1=O(t|\ln t|^\beta)$ if $\beta>1$.

\item The preceeding estimate actually holds with $\beta=1$. The key idea is to appeal directly to Lemma \ref{lem:main_lem}. In the rest of the discussion we fix $b=\frac32$. Now Lemma \ref{lem:main_lem} implies
\[
\omega(f,h)_1\leq2\int_0^1\sum_{l_{y,k}\geq h} h+\sum_{l_{y,k}<h} l_{y,k}\,dy.
\]
Recall that if $1/(k+1)< y<1/k$, $E_y$ is a union of $k$ intervals. If an interval $I_j$ belong to $E_{y^*}$ for some $y^*$ then $I_j\in E_{y}$ for $y\leq y^*$. Thus once an interval ``appears" it persists though, of course, its length changes with $y$. Indexed in order of appearance, the interval lengths are given by  $l_{y,1}=\frac34(1-y)$ and $l_{y,j}=\frac{j}{j^2-1}(\frac1j-y)$ for $2\leq j\leq k$.  Moreover, $l_{y,j}\geq h$ exactly when $h(k^2-1)<1$ and  $y<\frac1j-\frac{(j^2-1)h}{j}$ both hold.

It follows from this admitedly sketchy discussion that
\[
\omega(f,h)_1\leq I +II +III
\] 
where 
\begin{align*}
I&:=2h\int_0^{1-\frac{4h}{3}}dy+\frac32\int_{1-\frac{4h}{3}}^1(1-y)\,dy\\
II&:=2\sum_{k=2}^{(1+\frac1h)^{\frac12}}\brac{h\int_0^{\frac1k-\frac{(k^2-1)h}{k}}\,dy+\frac{k}{k^2-1}\int_{\frac1k-\frac{(k^2-1)h}{k}}^{\frac1k}\paren{\frac1k-y}\,dy}\\
III&:=2\sum_{k=(1+\frac1h)^{\frac12}}^\infty\frac{k}{k^2-1}\int_0^{\frac1k}\paren{\frac1k-y}\,dy 
\end{align*}
Fairly elementary computations now show that
\[
\omega(f,h)_1\leq 2h\sum_{k=1}^{(1+\frac1h)^{\frac12}}\frac1k+\sum_{k=(1+\frac1h)^{\frac12}}^\infty\frac{1}{k(k^2-1)}\leq h\ln\frac1h+O\paren{h}.
\]
\item At serious risk of pedantry, let us compute the $L^1$ modulus of continuity of $g(x)=|x\sin(\pi/x)|$ using the usual definition.
\begin{align*}
\omega(g,h)_1&=\int_0^1\abs{|(x+h)\sin(\pi/(x+h))|-|x\sin(\pi/x)|}\,dx\\
&\leq\int_0^1\abs{x(\sin(\pi/(x+h))-\sin(\pi/x)}\,dx+h\\
&\leq2\int_0^1\abs{x\sin\paren{\frac{\pi h}{2x(x+h)}}}\,dx+h\\
&\leq\int_0^1\frac{\pi h}{x+h}\,dx+h\leq\pi h\ln\paren{\frac{1+h}{h}}+h=O(h|\ln h|),
%&\leq\pi h\ln\frac1h+h\paren{1+\pi\ln(1+h)}
\end{align*}
which agrees with preceding computation.
\end{enumerate}
\end{eg}

\begin{eg}
As in {\sc Terekhin}, \cite[p.\,663]{Terekhin:1967ss}, we define $T(x)$ on $[0,1]$ as follows.
\begin{minipage}{0.5\textwidth}
Its graph is given by the equal sides of an isosceles triangles of unit height and base the interval $[2^{-k}, 2^{1-k}]$  for $k\in\N$ with $T(0)=0$, say. Terekhin shows that $\omega(T,t)_p=O(t^\frac1p)$ for $p>1$ and that $T(x)$ is not of bounded $q$--variation for any $q$.
\end{minipage}
\qquad
\begin{minipage}{0.5\textwidth}
\includegraphics[scale=.555]{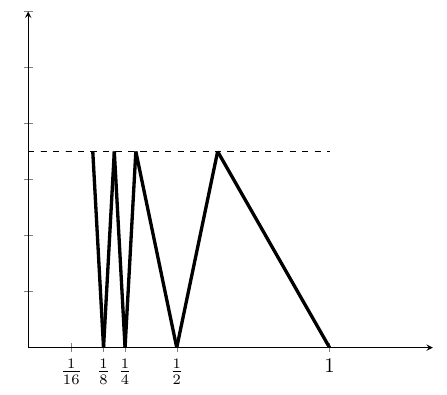}
\end{minipage}
Strictly speaking this function lies outside of the scope of Theorem \ref{thm:intro_thm}: it is continuous on the interval $[\epsilon,1]$ for any $\epsilon>0$ but is not continuous at 0. For $0\leq y\leq1$, $n(y)=\infty$ so this function is not of bounded variation. However, for $0\leq y<1$, $E_y$ is still an open set whose component intervals have length $l_{y,k}=2^{-k}(1-y)$ for $k\in\N$. The endpoints of these intervals have 0 as their only accumulation point.

As in the earlier example we take $\phi(t)=t^{-\alpha}$, where $0\leq\alpha<1$ and get that
\[
\int_0^1\sum_kl_{y,k}\phi(l_{y,k})\,dy=\int_0^1(1-y)^{1-\alpha}\sum_{k=1}^\infty2^{-k(1-\alpha)}\,dy<\infty.
\]
Hence $\omega(T,t)_1=O(t^{\alpha})$ for any $0\leq\alpha<1$. This suggests that the converse of \eqref{eqn:indicatrix_BVp} is false.

We claim that $\omega(T,t)_p=O(t^\frac1p|\log t|^\frac1p)$ for $p\geq1$. For $p>1$ this is just shy of Terekhin's result. For $p=1$, this estimate might well be new. To verify our claim, we once again use \eqref{eqn:main_sum} of Lemma \ref{lem:main_lem} and estimate
\[
\tau(h,E_y)\leq2\sum_{l_{y,k}\geq h} h+2\sum_{l_{y,k}<h} l_{y,k}\leq 2h\log_2\tfrac{1-y}{h}+ 2\sum_{k=\log_2\frac{1-y}{h}}^\infty2^{-k}(1-y).
\]
Thus 
\[\int_0^1\tau(h,E_y)\,dy\leq A\int_0^1\paren{h+h\log\tfrac{1-y}{h}}\,dy=O(h|\log h|),\]
and the claim follows from Proposition \ref{prop:main_reduc} and H\"older's inequality.
\end{eg}

\begin{eg}[Fat Cantor Set]
We briefly recall the construction of so called ``Fat Cantor sets" (FCS).

\begin{minipage}{0.5\textwidth}
Fix $0<\lambda<\frac13$ and remove the middle open portion of $[0,1]$ of length $\lambda$. The construction continues by removing the open middle portion of length $\lambda^2$ from the remaining two closed subintervals, and so on. At stage $k$ we remove  $2^{k-1}$ intervals each of length $\lambda^k$. The limiting set, $K$, has measure $\frac{1-3\lambda}{1-2\lambda}\in (0,1)$
\end{minipage}
\quad
\begin{minipage}{0.4\textwidth}
\includegraphics[scale=.455]{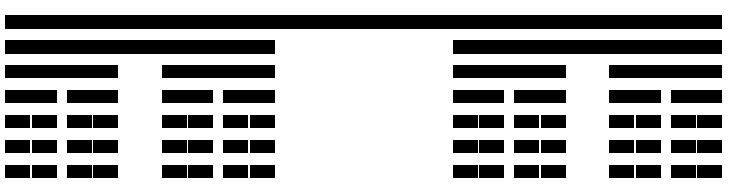}
\end{minipage}

The FCS has empty interior and could \emph{never} be the complement of the superlevel set of a continuous function. Indeed if $0<y<1$ and $K_y=\{x:f(x)\leq y\}$ is the complement of $E_y$, then $K_y$ contains at least one closed non-degenerate interval as a result of the continuity of $f$.

We will apply Theorem \ref{thm:main_thm} to the open set $E$ consisting of all the intervals removed. For $0\leq\alpha<1$ and $\beta\geq0$ let $\phi(t)=t^{-\alpha}\paren{\log\frac1t}^\beta$ it follows that
\[
\sum_{j} l_j\phi(l_j)=\tfrac12|\log\lambda|^\beta\sum_j(2\lambda^{1-\alpha})^jj^\beta
\]
This sum converges as long as $2\lambda^{1-\alpha}<1$ or when $\alpha<1-\frac{\log 2}{\log\frac1\lambda}$. Hence 
\[
\tau(h,E) =O_{\lambda,\epsilon}(h^{1-\frac{\log 2}{\log\frac1\lambda}-\epsilon}),\quad \epsilon>0;
\]
where the implied constant depends on $\lambda$ and $\epsilon$.

We now show that the exponent of $h$ can be chosen equal to $1-\frac{\log 2}{\log\frac1\lambda}=1-d_X(E)$.

Given $h>0$, choose $m$ such that $\lambda^{m+1}< h\leq\lambda^m$ so, roughly, $m\sim\frac{\log_2h}{\log_2\lambda}$.  Using Lemma \ref{lem:main_lem} we have
\[
\tau(h,E)\leq2\sum_{l_j\geq h} h+2\sum_{l_j<h} l_j \leq 2h\sum_{j=0}^m2^j+\sum_{j=m+1}^\infty 2^{j-1}\lambda^j.
\]
Summing a geometric series and some algebra now gives
\[
\tau(h,E)\leq h2^{m}\tfrac{(3-4\lambda)}{1-2\lambda}\leq\tfrac{(3-4\lambda)}{1-2\lambda}\cdot h^{1-\frac{\log2}{\log\frac1\lambda}}.
\]
\end{eg}
\begin{eg}[FCS Revisited]
The estimate in Example 3 is sharp as we now show. The ensuing argument uses, as it must, the actual structure of $K$. Since
\[
\tau(h,E) =\sum_k\paren{\mes\{x\in I_k:x+h\in K\} +\mes\{x\in I_k:x-h\in K\}},
\]
we only need to show that as $h\to0^+$ there is a small positive number $a$, independent of $h$, such that
\[
\sum_{l_k\geq h}\mes\{x\in I_k:x+h\in K\} > a h^{1-\frac{\log2}{\log\frac1\lambda}}.
\]
Let $K_m$ be the Cantor set after the $m$-th stage. By construction, $K_m=\bigcup J^m_k$ is a finite union of closed intervals of equal length given by
\[l(J^m_k)=2^{-m}\paren{1-\sum_{j=1}^m2^{j-1}\lambda^j}=2^{-m}\paren{1-\frac{\lambda(1-(2\lambda)^m)}{1-2\lambda}}.\]

 Assume $h\leq\lambda$ and, as before, let $m$ satisfy $\lambda^{m+1}< h\leq\lambda^m$.
Each interval $I_k$, from $E$ with $l(I_k)\geq h$ touches two consecutive intervals $J_{l-1}^m$ and $J_l^m$ in $K_m$. Because $\lambda^m<l(J^m_l)$ and $l(I_k)\geq h$, if $I_k$ abuts $J_{l-1}^m$on the right then $I_k+h$ intersects $J_{l-1}^m$ in an interval of length $h$. We still need to determine the measure $\abs{\{x\in I_k:x+h\in K\}}$. 
\begin{center}
\includegraphics[scale=.85]{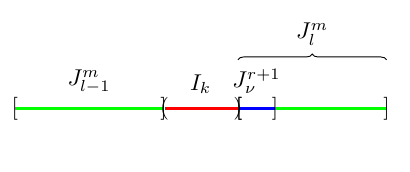}
\end{center}
To this end we choose $r\geq m$ so that at the $r$th stage $J^{r+1}_\nu\subset I_k+h\subset J^r_l$ for some $\nu$ (see above figure). In particular $r$  will be the largest integer for which
\[2^{-r}\paren{1-\frac{\lambda(1-(2\lambda)^r)}{1-2\lambda}}\geq h\]
If $A=\frac{\lambda}{1-2\lambda}$ and $B=\frac{1-3\lambda}{1-2\lambda}$ and $\gamma=\log_2\frac1\lambda$ then, essentially, $2^{-r}$ solves
\[Az^\gamma+Bz-h=0.\]
If $J_{\nu_k}^{r+1}\subset I_k+h\subset J_{l_k}^{r}$ are prescribed as above, we can determine $|K \cap J_\nu^{r+1}|$ by successively removing middle intervals of length $\lambda^{r+2},\lambda^{r+3},\lambda^{r+4},\ldots,$ from $J^{r+1}_\nu$. Consequently,
\begin{align*}
\sum_{l(I_k)\geq h} \mes\{x\in I_k:x+h\in K\}\geq\sum_{k=0}^{m-1}|K \cap J_{\nu_k}^{r+1}|&\geq (2^m-1)\cdot2^{-r-1}B.
\end{align*}
Since $\lambda^{m+1}< h\leq\lambda^m$, some algebra shows that $2^m\geq\frac12h^{-\frac1\gamma}$ which means $2^m-1\geq\frac14h^{-\frac1\gamma}$.

It remains to estimate $2^{-r}$ from below. This requires knowing the asymptotic behavior of solutions to $Az^\gamma+Bz-h=0$ for small $h$. To handle this, write $z=sh/B$ so that the equation becomes $\frac{h^{\gamma-1}A}{B^\gamma}s^\gamma=1-s$. Since $s>0$, the left side is positive so $s<1$ must hold. This in turn yields $s>1-h^{\gamma-1}A/B^\gamma$ and from this follows $2^{-r}>\frac hB(1-h^{\gamma-1}A/B^\gamma)$.
Hence
\[\tau(h,E)\geq (2^m-1)\cdot2^{-r-1}B\geq\tfrac18(1-h^{\gamma-1}A/B^\gamma)h^{1-\frac{\log2}{\log\frac1\lambda}}\geq \tfrac{1}{16}h^{1-\frac{\log2}{\log\frac1\lambda}},\]
as long as  $h\leq\paren{\frac{B^\gamma}{2A}}^{\frac{1}{\gamma-1}}$.
\end{eg}
\section{Generalizations}
We sketch simple higher dimensional variants of some ideas from \S3.

Let $E$ be an arbitrary, non-empty bounded open subset of $\R^d$. To be concrete we assume $E\subset[0,1]^d$. Let $\Gamma=\overline{E}\backslash E$ denote its boundary and $K=\R^d\backslash E$ its complement. If $h>0$ we also define the $h$-neighborhoods 
\[
K(h)=\{x:\dist(x,K)\leq h\};\qquad
\Gamma(h)=\{x:\dist(x,\Gamma)\leq h\}.
\]
If $\vec{v}$ is a unit vector, we are interested in
\[\tau(h,\vec{v},E)=\mes\{x\in\R^d:|\chi_E(x+h\vec{v})-\chi_E(x)|\neq0\}.\]
Again, we have
\[
\{x\in\R^d:|\chi_E(x+h\vec{v})-\chi_E(x)|\neq0\}\subset K(h)\backslash K\subset\Gamma(h),
\]
which in turn implies that
\[\tau(h,\vec{v},E)\leq| K(h)\backslash K|\leq|\Gamma(h)|.\]
Define the Besicovitch--Taylor index
\[d_X(E)=d-\lim_{h\to0}\frac{\ln| K(h)\backslash K|}{\ln h},\]
and the Box-counting dimension
\[d_B(\Gamma)=d-\lim_{h\to0}\frac{\ln|\Gamma(h)|}{\ln h},\]
assuming for simplicity that both limits exist. The inclusion relation makes it easy to see that
\[d_X(E)\leq d_B(\Gamma).\]
From here it follows that given any $\epsilon>0$ there is a $\delta>0$ such that 
\[\tau(h,\vec{v},E)\leq h^{d-d_X-\epsilon};\qquad0<h<\delta\]
This in turn shows that $\omega(\chi_E,t)_1=O(t^{d-d_X-\epsilon})$. Related estimates are given by {\sc Sickel} \cite{Sickel:2021fp} who treats $E$ that are bounded, open and \emph{connected}. For the majority of domains considered in \cite{Sickel:2021fp}, $d_X(E)$ and $d_B(\Gamma)$ coincide as far as we can tell.
\section{Loose Ends}
What follows is a hodgepodge list of curious observations, half-baked ideas and naive questions related to our theme. They would have muddled the narrative of the other sections, but I still mention them in the off chance they spark something in your mind\footnote{If this happens, I would love to hear your thoughts.}.
\begin{enumerate}[(a)]
\item What properties of $E$ would improve our estimate for $\tau(h,E)$? \emph{Density}? \emph{Porosity}? \emph{Thinness/Thickness}? \emph{Capacity}?
\item A continuous nowhere differentiable function cannot be of bounded variation since $N(E_y)=\infty$ for $y$ in a set of positive measure. It would be nice to apply Theorem \ref{thm:intro_thm} to a concrete function of this type. The sample paths of various stochastic processes are the obvious dragons to slay here, but there is a whole zoo to choose from.
\item The superlevel sets satisfy the inclusion relation $E_{y_2}\subset E_{y_1}$ whenever $0\leq y_1\leq y_2\leq1$. How does the structure of $E_y$ change when $y$ is varied slightly?
\item  Given a partition $y_1<y_2<\ldots<y_M$ we have the decomposition $E_{y_i}=\bigcup I_{y_i,k}$. The intervals $\{I_{y_i,k}:1\leq k\leq N(E_{y_i});\,\,1\leq i\leq M\}$
essentially form a tree under inclusion. What does the structure of this tree say, if it says anything, about the function?
\item 
In Theorem \ref{thm:intro_thm} we used a single ``gauge function", $\phi$ for all the level sets $E_y$. Would allowing a $y$ dependence sharpen things? What about measurability issues?
\item Earlier we mentioned the logarithmic measure of $E$, namely $\sum l_k|\log{l_k}|$. {\sc Bary} \cite[p.\,421--428]{Bary:1964sj} showed that if an open subset of the circle has finite logarithmic measure, then the Fourier series of its characteristic function converges almost everywhere. The same is true of a function for which $E_y$ has finite logarithmic measure for a.e.\,$y$. Both conclusions are rendered obsolete by the Carleson--Hunt theorem. However, is there an analog of Bary's result for the spherical partial sums of characteristic functions of open subsets of the unit square?
\item For $d\geq2$ and a bounded open set $E\subset\R^d$, does the Whitney decomposition of $E$ give more precise estimates for $\tau(h,\vec{v},E)$ of \S5? For randomly chosen $\vec{v}$, what is the ``most likely" behaviour of $\tau(h,\vec{v},E)$ and how does it depend on $E$?
%\item Is this problem a haystack with Buffon's needle somewhere in there?
\end{enumerate}
%These are the ideas I felt not too crazy to appear in print. Thanks for making it this far.
%\newpage
\bibliography{lib_papers} %This is the global bib library
\bibliographystyle{amsplain}

\end{document}